\documentclass[reqno]{amsart}

\usepackage{amssymb, amsmath, amsthm, amsfonts,graphicx}

\usepackage{mathrsfs}

\def\S{{\mathcal S}}

\usepackage{array}

\theoremstyle{plain}
\newtheorem{theorem}{Theorem}
\newtheorem{corollary}[theorem]{Corollary}

\theoremstyle{definition}
\newtheorem{definition}{Definition}

\newtheorem{remark}[theorem]{Remark}

\numberwithin{equation}{section}
\numberwithin{theorem}{section}
\numberwithin{definition}{section}
\numberwithin{prop}{section}
\numberwithin{example}{section}

\newcommand{\R}{\ensuremath{\mathbb{R}}}
\newcommand{\Z}{\ensuremath{\mathbb{Z}}}
\def\H{{\mathcal H}}
\newcommand{\N}{\ensuremath{\mathbb{N}}}

\newcommand{\T}{\ensuremath{\mathbb{T}}}
\newcommand{\C}{\ensuremath{\mathbb{C}}}

\DeclareMathOperator{\myRe}{Re}

\DeclareMathOperator{\spec}{spec}

\renewcommand{\Re}{\myRe}

\newcommand{\s}{\sigma}
\newcommand{\m}{\mu}

\newcommand{\D}{\Delta}

\begin{document}


\title[Algebraic and Dynamic Lyapunov Equations on Time Scales]{Algebraic and Dynamic Lyapunov Equations\\ on Time Scales*}
\author[Davis, Gravagne, Marks, Ramos]{John M. Davis$^1$, Ian A. Gravagne$^2$, Robert J. Marks II$^2$, Alice A. Ramos$^3$}
\address{$^1$Department of Mathematics, Baylor University, Waco, TX 76798}
\email{John\_M\_Davis@baylor.edu}
\address{$^2$Department of Electrical and Computer Engineering, Baylor University, Waco, TX 76798}
\email{Ian\_Gravagne@baylor.edu,  Robert\_Marks@baylor.edu}
\address{$^3$Department of Mathematics, Bethel College, Mishawaka, IN 46545}
\email{alice.ramos@bethelcollege.edu}

\thanks{*This work was supported by NSF Grant CMMI\#726996. See {\tt http://www.timescales.org/} for other papers from the Baylor Time Scales Research Group.}

\begin{abstract}
We revisit the canonical continuous-time and discrete-time matrix algebraic and matrix differential equations that play a central role in Lyapunov based stability arguments. The goal is to generalize and extend these types of equations and subsequent analysis to dynamical systems on domains other than $\R$ or $\Z$, e.g. nonuniform discrete domains or domains consisting of a mixture of discrete and continuous components. We compare and contrast the standard theory with the theory in this general case.
\end{abstract}

\maketitle
\tableofcontents

\section{Lyapunov Equations and Stability}
%
One of the most widely used tools for investigating the stability of linear
systems is the Second (Direct) Method of Lyapunov, presented in his
dissertation of 1892. The idea of this method is to
investigate stability of a given system by measuring the rate of
change of the energy of the system. The advantage of this
approach is that it allows one to infer the
stability of differential (and difference) equations without
explicit knowledge of solutions.

We begin with a review of Lyapunov's Second (Direct) Method in
the context of linear differential equations on $\R$ and linear
difference equations on $\Z$.  Then, building on the work of DaCunha
\cite{Da2} we proceed to unify and extend this well known
theory for application to dynamic linear systems defined on
arbitrary time scale domains.

Specifically, DaCunha extended Lyapunov's Second (Direct) Method for
application to stability analysis of certain classes of dynamic
systems (e.g., slowly time varying systems) and developed and solved
a time scale algebraic Lyapunov equation.  Here we develop and
solve a {\em dynamic} time scale Lyapunov equation which has
application to the stability analysis of a much broader class of
systems.

\section{Stability of Continuous-time Systems}\label{R}
We begin by considering the familiar linear state
equation
    \begin{equation}\label{CLS}
        \dot{x}(t)= A(t) x(t),
    \end{equation}
for $A\in\R^{n\times n}$, and $t\in\R$. We assume that \eqref{CLS}
has equilibrium $x=0$. There are various notions of stability of solutions to \eqref{CLS}, which we outline now. Without loss of generality, we take the origin as the reference equilibrium.

\begin{definition}
An equilibrium $x=0$ of \eqref{CLS} is \textit{Lyapunov stable} or
\textit{stable in the sense of Lyapunov} if, for every
$\varepsilon>0$, there exists a $\delta=\delta(\varepsilon)>0$ such
that if $||x(t_0)||<\delta$, then $||x(t)||<\varepsilon$ for every
$t \geq t_0$.
\end{definition}


\begin{definition}
    An equilibrium $x=0$ of \eqref{CLS} is called \textit{uniformly
    stable} if there exists a finite positive constant $\gamma$ such
    that for any $t_0$ and $x_0$ the corresponding solution satisfies
    $$||x(t)|| \leq \gamma ||x_0||,\qquad t\geq t_0.$$
\end{definition}

    \begin{definition}\label{as_R}
    An equilibrium $x=0$ of \eqref{CLS} is \textit{asymptotically
    stable} if it is Lyapunov stable and there exists a $\delta>0$ such
    that if $||x(t_0)||<\delta$, then $\lim_{t\rightarrow\infty} ||x(t)|| =
    0$.
    Furthermore, an equilibrium $x=0$ of \eqref{CLS} is \textit{uniformly
    asymptotically stable} if it is uniformly stable and if given any
    $\delta > 0$, there exists a $T$ such that for any
    $t_0$ and $x_0$ the corresponding solution satisfies
    $$
    ||x(t)||<\delta||x_0||,\qquad t \geq t_0+T.
    $$
    \end{definition}
    \begin{definition}
    An equilibrium $x=0$ of \eqref{CLS} is \textit{exponentially
    stable} if it is asymptotically stable and there exist constants
    $\gamma,\lambda,\delta > 0$ such that if $||x(t_0)||<\delta$, then
        $$
        ||x(t)||\leq\gamma e^{-\lambda(t-t_0)}||x(t_0)||,\qquad t \geq t_0.$$
    Furthermore, $x=0$ of \eqref{CLS} is \textit{uniformly
    exponentially stable} if there exist $\gamma,\lambda >0$
    such that for any $t_0$ and $x_0$ the corresponding solution
    satisfies
        $$
        ||x(t)||\leq\gamma e^{-\lambda(t-t_0)}||x(t_0)||,\qquad t\geq t_0.
        $$
    \end{definition}
It is important to note that, in the context of linear systems,
uniform exponential stability and uniform asymptotic stability are
equivalent. That is, an equilibrium is uniformly exponentially
stable if and only if it is uniformly asymptotically stable. For a
straightforward proof of this, see Rugh \cite[Theorem~6.13]{Rugh}.

The above definitions characterize stability in terms of the
boundedness and convergence of solutions.  In order to establish
stability without explicit knowledge of such solutions, we look to
Lyapunov's Second (Direct) Method.

%
%

\begin{definition}\label{LyapfnR}
    $V:\R^n\rightarrow\R$ is a \textit{Lyapunov
    function} for \eqref{CLS} if
        \begin{itemize}
        \item [(i)]$V(x(t))\geq 0$ with equality if and only if $x=0$, and
        \item [(ii)]$\dot{V}(x(t))\leq 0.$
        \end{itemize}
\end{definition}


This leads to the celebrated theorem of A.M. Lyapunov \cite{Ly}.

    \begin{theorem}[Lyapunov's Second Theorem on $\R$]\label{LyapR}
    Given system \eqref{CLS} with equilibrium $x=0$, if there exists
    an associated Lyapunov function $V$, then $x=0$ is Lyapunov stable.
    Furthermore, if $\dot{V}(x(t))< 0$, then $x=0$ is asymptotically stable.
    \end{theorem}

The power of this theorem is that one can deduce stability properties of solutions without direct knowledge of the solutions. We will discuss some techniques for finding suitable Lyapunov functions in the next few subsections.

%
\subsection{The Continuous-time Algebraic Lyapunov Equation}
To satisfy the requirements of Theorem~\ref{LyapR} for asymptotic
stability, we seek a Lyapunov function, $V$, such that (a)
$V(x)>0$, with $\dot{V}(x)<0$ for $x\neq 0$, and (b)
$V(0)=\dot{V}(0)=0$.  For the system \eqref{CLS}, a common choice of
Lyapunov function candidate is the quadratic form, $V(x(t))=x^T(T)Px(t)$. We begin investigation of the stability of \eqref{CLS} by considering the time derivative of $V(x)$:
    \begin{align*}
    \dot{V}(x(t))&= \frac{d}{dt}[x^T(t)P x(t)]\\
         &= x^T(t)P\dot{x}(t) + \dot{x}^T(t) Px(t)\\
         &= x^T(t)PA(t)x + (A(t)x)^T Px(t)\\
         &= x^T(t)PA(t)x + x^T A^T(t) Px(t)\\
         &= x^T(t)[A^T(t)P + PA(t)]x(t).
    \end{align*}
The quadratic form of this derivative proves useful because if the
central quantity satisfies
    $$A^T(t) P + PA(t)<0,$$
then $\dot{V}(x(t))<0$.  Thus, for our purposes it is sufficient to seek a $P\in\S^+_n$ which satisfies the \textit{algebraic Lyapunov equation}
    \begin{equation}\label{CALE}
         A^T(t) P + PA(t)= -M(t),\tag{CALE}
    \end{equation}
where $M(t)\in\S_n^+$ is given. Here, $\S_n^+$ ($\S_n^-$) denotes the set of real, $n\times n$ positive (negative) definite symmetric matrices.

To distinguish this from other Lyapunov equations, we refer to \eqref{CALE} as
the \textit{continuous-time algebraic Lyapunov equation}. The
following theorem establishes a closed form solution of \eqref{CALE}.

\begin{theorem}\textup{\cite{Riccati,Rugh}}\label{CALEsolthm}
The unique solution of
    \begin{equation*}
        A^T(t) P + PA(t)  = -M(t),
    \end{equation*}
is given by
    \begin{equation}\label{CALEsol}
    P(t)=\int^\infty_{t_0} \Phi_A^T(s,t_0)M(t)\Phi_A(s,t_0)\,ds,
    \end{equation}
where $\Phi_A(t,t_0)$ is the transition matrix for system \eqref{CLS}, i.e., $\Phi_A(t,t_0)$ solves
    $$\dot{X}(t)=A(t)X(t), \qquad X(t_0)=I.$$
Moreover, $P\in\S_n^+$ whenever $M(t)\in\S_n^+$.
\end{theorem}

\begin{corollary}
When $A(t)\equiv A$ and $M(t)\equiv M$, the unique solution of \eqref{CALE}
is the constant
    $$
    P=\int^\infty_{t_0} e^{A^T(s-t_0)}Me^{A(s-t_0)}\,ds,
    $$
which converges when $\Re\lambda<0$ for all $\lambda\in\spec A$.
\end{corollary}

Note that it is the quadratic form of the solution that allows us to
conclude $M\in\S^+_n$ implies $P\in\S^+_n$. It follows that existence of the quantity in
\eqref{CALEsol} implies the existence of a Lyapunov function
satisfying the requirements of Theorem~\ref{LyapR} and ensures
stability of system \eqref{CLS}.  Thus, \eqref{CALE}
straightforwardly leads to a viable Lyapunov function that we seek for a stability analysis of the underlying system \eqref{CLS}.

\subsection{The Continuous-time Differential Lyapunov Equation}
On the other hand, suppose we seek a Lyapunov function of the form $V(x(t))=x^T(t) P(t)
x(t)$, the emphasis being that $P$ is time varying. Then
    $$
    \begin{aligned}
    \dot{V}(x(t))&= \frac{d}{dt}[x^T(t) P(t) x(t)]\\
         &= x^T(t)[P(t)\dot{x}(t) + \dot{P}(t)x] + \dot{x}^T P(t)x(t)\\
         &= x^T(t)[P(t)A(t)x +\dot{P}(t)x] + (A(t)x)^T P(t)x(t)\\
         &= x^T(t)[P(t)A(t)x +\dot{P}(t)x] + x^T A^T(t) P(t)x(t)\\
         &= x^T(t)[A^T(t) P(t) + P(t)A(t) + \dot{P}(t)]x(t).
    \end{aligned}
    $$
The quadratic form of this derivative proves useful because, if the
central quantity satisfies
    $$A^T(t) P(t) + P(t)A(t) + \dot{P}(t)<0,$$
then $\dot{V}(x)<0$.  Thus we seek a $P(t)\in\S^+_n$ which satisfies the \textit{differential Lyapunov equation}
    \begin{equation}\label{CDLE}
         A^T(t) P(t) + P(t)A(t) +\dot{P}(t)= -M(t),\tag{CDLE}
    \end{equation}
where $M(t)\in\S^+_n$ is specified.

To distinguish this from other
Lyapunov equations, we refer to \eqref{CDLE} as
the \textit{continuous-time differential Lyapunov equation}.  The
following theorem establishes a closed form solution of
\eqref{CDLE}.

\begin{theorem} \textup{\cite{Riccati}}
The unique solution of
    \begin{equation*}
         A^T(t) P(t) + P(t)A(t) +\dot{P}(t)= -M(t),\quad P(t_0)=P_0,
    \end{equation*}
is given by
    \begin{equation}\label{CDLEsol}
        P(t)=(\Phi_A^T(t,t_0))^{-1}P(t_0)(\Phi_A(t,t_0))^{-1}-\int^t_{t_0}
            \Phi_A^T(t,s)M(s)\Phi_A(t,s)\,ds,
    \end{equation}
where $\Phi_A(t,t_0)$ is the transition matrix for system \eqref{CLS}. Moreover, $P(t)\in\S_n^+$ whenever $M(t)\in\S_n^+$.
\end{theorem}

\begin{remark}
From the derivation of \eqref{CALE}, we see that the constant
solution of \eqref{CALE} is in fact a steady state solution of \eqref{CDLE} provided the initial condition $P(t_0)$ is chosen to be this constant solution of \eqref{CALE}.
\end{remark}


In light of the remark above, it is not surprising that \eqref{CALE} takes precedence over \eqref{CDLE} in the literature---matrix {\it algebraic} equations are much easier to solve than matrix {\it differential} equations. Even so, \eqref{CDLE} is an interesting problem in its own right. There are problems in which the time varying nature of a system makes \eqref{CDLE} useful, especially in the context of periodic systems \cite{Riccati,DeMa}. However, when the time dependence is not of interest, or has minimal impact on the system (e.g., slowly time varying systems), it is more efficient to consider the algebraic equation \eqref{CALE} to obtain simpler, steady state solutions of the differential equation \eqref{CDLE}.

\section{Stability of Discrete-time Systems}\label{Z}
We now turn our attention to the discrete analogue of the continuous
system \eqref{CLS} analyzed in the preceeding section.  Let $t\in\Z$
and consider the discrete linear system
    \begin{equation}\label{DLS}
        \Delta x(t) = A(t)x(t),
    \end{equation}
for $A\in\R^{n\times n}$, and $t\in\Z$, where $\Delta x(t):=x(t+1)-x(t)$ is the
usual forward difference operator. Rearranging \eqref{DLS} and defining
$A_R(t):=A(t)+I$, we can write \eqref{DLS} in its
(possibly more familiar) equivalent recursive form
    \begin{equation*}
        x(t+1)= A_R(t) x(t).
    \end{equation*}

We begin with several characterizations of stability. We include these in order to compare our generalized results in the next section to the standard results in the continuous and discrete settings.


    \begin{definition}
    Let $t\in\Z$. An equilibrium $x=0$ of \eqref{DLS} is \textit{Lyapunov stable} or
    \textit{stable in the sense of Lyapunov} if, for every
    $\varepsilon>0$, there exists a $\delta=\delta(\varepsilon)>0$ such that
    if $||x(t_0)||<\delta$, then $||x(t)||<\varepsilon$, for every $t \geq
    t_0$.  An equilibrium $x=0$ of \eqref{DLS} is called \textit{uniformly
    stable} if there exists a finite positive constant $\gamma$ such
    that for any $t_0$ and $x_0$ the corresponding solution satisfies
    $$
    ||x(t)|| \leq \gamma ||x_0||,\qquad t\geq t_0.
    $$
    \end{definition}
    \begin{definition}
    Let $t\in\Z$. An equilibrium $x=0$ of \eqref{DLS} is \textit{asymptotically
    stable} if it is Lyapunov stable and there exists a $\delta>0$ such
    that if $||x(t_0)||<\delta$, then $\lim_{t\rightarrow\infty} ||x(t)|| =
    0$. Furthermore, an equilibrium $x=0$ of \eqref{DLS} is \textit{uniformly
    asymptotically stable} if it is uniformly stable and if given any
    positive constant $\delta$ there exists a $T$ such that for any
    $t_0$ and $x_0$ the corresponding solution satisfies
    $$
    ||x(t)||<\delta||x_0||,\qquad t \geq t_0+T.
    $$
    \end{definition}
    \begin{definition}
    Let $t\in\Z$. An equilibrium $x=0$ of \eqref{DLS} is \textit{exponentially
    stable} if it is asymptotically stable and there exist constants
    $\gamma,\delta>0$ and $0 \leq \lambda < 1$ such that if $||x(t_0)||<\delta$,
    then
    $$
    ||x(t)||\leq||x(t_0)||\gamma \lambda^{t-t_0}, \qquad t \geq t_0.
    $$
    Furthermore, $x=0$ of \eqref{DLS} is \textit{uniformly
    exponentially stable} if there exist finite positive constants
    $\gamma$ and $0 \leq \lambda < 1$ such that for any $t_0$ and $x_0$ the
    corresponding solution satisfies
    $$
        ||x(t)||\leq||x(t_0)||\gamma \lambda^{t-t_0},\qquad t\geq
        t_0.
    $$
    \end{definition}

With the goal of analyzing stability of \eqref{DLS} without explicit
knowledge of solutions, we explore the application of Lyapunov's
Second Method in the context of discrete-time systems.

    \begin{definition}\label{LyapfnZ}
    $V(x):\R^n\rightarrow\R$ is a \textit{Lyapunov
    function} for system \eqref{DLS} if
        \begin{itemize}
        \item [(i)]$V(x)\geq 0$ with equality if and only if $x=0$, and
        \item [(ii)]$\Delta V(x(t))\leq 0.$
        \end{itemize}
    \end{definition}

    \begin{theorem}[Lyapunov's Second Theorem on $\Z$]\label{LyapZ}
    Given system \eqref{DLS} with equilibrium $x=0$, if there exists
    an associated Lyapunov function $V(x)$, then $x=0$ is Lyapunov stable.
    Furthermore, if $\Delta V(x(t))< 0$, then $x=0$ is asymptotically stable.
    \end{theorem}

\subsection{The Discrete-time Algebraic Lyapunov Equation}
We begin with the same choice of Lyapunov function candidate
utilized in the continuous case, $V(x)=x^T(t) Px(t)$. Then
    \begin{align*}
        \Delta V(x(t))&= V(x(t+1))-V(x(t))\\
            &= x^T(t+1)Px(t+1)- x^T(t)Px(t)\\
            &= x^T(t)(A(t)+I)^T P(A(t)+I)x(t) -
                x^T(t)Px(t)\\
            &= x^T(t) [(A(t)+I)^T P(A(t)+I) - P] x(t)\\
            &= x^T(t) [ A(t)P + PA(t) + A(t)PA(t)]x(t).
    \end{align*}
If the central quantity satisfies
    $$
    A(t)P+PA(t)+A(t)PA(t)<0,
    $$
then $\Delta V(x)<0$.  Therefore,
we seek a $P(t)\in\S^+_n$ satisfying the \textit{discrete-time algebraic Lyapunov equation},
    \begin{equation}\label{DALE}
    A(t)P + P A(t) + A(t)PA(t)=-M(t),\tag{DALE}
    \end{equation}
for a given $M(t)\in\S^+_n$.

Equivalently, \eqref{DALE} also has the recursive form
    \begin{equation}\label{DALEr}
        A^T_R(t) P A_R(t) - P = -M(t), \tag{DALEr}
    \end{equation}
where $A_R(t):=A(t)+I$. This form seems to be more common in the literature on Lyapunov analysis of discrete linear systems \cite{NaBa,Rugh,TiMa}.

The following theorem establishes a closed form solution of \eqref{CALE}.

\begin{theorem}\textup{\cite{Riccati,Rugh}}
For $A(t)\equiv A$ and $M(t)\equiv M$, the unique solution of
\eqref{DALEr} is the constant
    \begin{equation*}
        P = \sum^\infty_{j=0} (A_R^T)^j M A_R^j.
    \end{equation*}
Moreover, $P\in\S_n^+$ whenever $M\in\S_n^+$. The sum converges provided $|\lambda|<1$ for all $\lambda\in\spec A$.
\end{theorem}

\subsection{The Discrete-time Difference Lyapunov Equation}
Again, if we start with a Lyapunov candidate of the form $V(x(t))=x^T(t) P(t) x(t)$, we obtain
    \begin{align*}
        \Delta V(x(t))&= V(x(t+1))-V(x(t))\\
            &= x^T(t+1)P(t+1)x(t+1)- x^T(t)P(t)x(t)\\
            &= x^T(t)(A(t)+I)^T P(t+1)(A(t)+I)x(t) -
                x^T(t)P(t)x(t)\\
            &= x^T(t) [(A(t)+I)^T P(t+1)(A(t)+I) - P(t)] x(t)\\
            &= x^T(t) [ A(t)P(t+1) + P(t+1)A(t) + A(t)P(t+1)A(t) + \Delta P(t) ]
            x(t).
    \end{align*}
If the central quantity satisfies
    $$
        A(t)P(t+1)+P(t+1)A(t)+A(t)P(t+1)A(t) + \Delta P(t)<0,
    $$
then $\Delta V(x)<0$.  Therefore,
we seek a $P(t)\in\S^+_n$ satisfying the \textit{discrete-time difference Lyapunov equation},
    \begin{equation}\label{DDLE}
        A(t)P(t+1) + P(t+1)A(t) + A(t)P(t+1)A(t) + \Delta P(t)=-M(t),\tag{DDLE}
    \end{equation}
for a given $M(t)\in\S^+_n$. This equation is more often seen in a
recursive form
    \begin{equation*}
         A_R^T(t) P(t+1) A_R(t) - P(t)=-M(t). \tag{DDLEr}
    \end{equation*}
To show stability of system \eqref{DLS} via Theorem~\ref{LyapZ}, we
begin by solving \eqref{DDLE} for $P(t)$.

\begin{theorem}\textup{\cite{Va}}
The unique solution of \eqref{DDLE} satisfying $P(t_0)=P_0$ is given by
    \begin{equation}\label{DDLEsol}
    P(t)=(\Phi_A^T(t,t_0))^{-1}P(t_0)(\Phi_A(t,t_0))^{-1}-\sum^t_{s=t_0}\Phi_A^T(s,t)M(s)\Phi_A(s,t),
    \end{equation}
where $\Phi_A(t,t_0)$ is the transition matrix for \eqref{DLS},
i.e., $\Phi_A(t,t_0)$ solves
$$ \Delta X(t) = A(t)X(t), \qquad X(t_0)=I.$$ Moreover, $P(t)\in\S_n^+$ whenever $M(t)\in\S_n^+$.
\end{theorem}

\begin{remark}
From the derivation of \eqref{DALE}, we see that the constant
solution of \eqref{DALE} is in fact a steady state solution of \eqref{DDLE} provided the initial condition $P(t_0)$ is chosen to be this constant solution of \eqref{DALE}.
\end{remark}


\eqref{DDLE} and its corresponding solution are most useful for
stability analysis of linear systems when the time dependent nature
of the equation is relevant. For example, \eqref{DDLE} is frequently
seen in the context of the analysis of discrete periodic systems
\cite{BiCoDe,Va,Va2}. However, when the time-dependent aspect is not
of interest (e.g., $A(t)\equiv A$ or $A$ slowly time varying), it
makes sense to simplify the problem (as we did in the continuous
case) to an algebraic problem by seeking steady state solutions of
\eqref{DDLE}.

\section{A Unified Approach to Lyapunov Stability}\label{T}
Now we turn our attention to generalizing the previous concepts from $\R$ and $\Z$ to more general time domains (time scales), e.g. nonuniform discrete sets or sets with a combination of discrete and continuous components. We begin with a brief overview of time scales and the prerequisite time scale calculus needed in order to examine generalized Lyapunov equations, their solutions, and properties of solutions in this framework.

\subsection{What Are Time Scales?}

The theory of time scales springs from the 1988 doctoral dissertation of Stefan Hilger \cite{Hi2} that resulted in his seminal paper \cite{Hi1}. These works aimed to unify various overarching concepts from the (sometimes disparate) theories of discrete and continuous dynamical systems \cite{MiHoLi}, but also to extend these theories to more general classes of dynamical systems. From there, time scales theory advanced fairly quickly, culminating in the excellent introductory text by Bohner and Peterson \cite{BoPe2} and the more advanced monograph \cite{BoPe1}. A succinct survey on time scales can be found in \cite{AgBoORPe}.

\begin{table}
\caption{Canonical time scales compared to the general case.}
\label{comparisons}
\renewcommand{\arraystretch}{1.25}
\begin{tabular}{|>{\centering}m{1.75cm}|>{\centering}m{1.5in}|>{\centering}%
  m{1.5in}|>{\centering\arraybackslash}m{1.5in}|}
  \hline
   & \sf{continuous} & \sf{(uniform) discrete} & \sf{time scale}\\ \hline
  \sf{domain} & $\R$ & $\Z$ & $\T$ \\ \hline
  & \rule{0ex}{3em}\includegraphics[scale=.35]{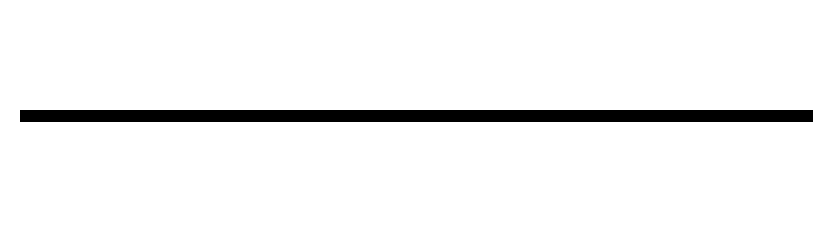} & \includegraphics[scale=.35]{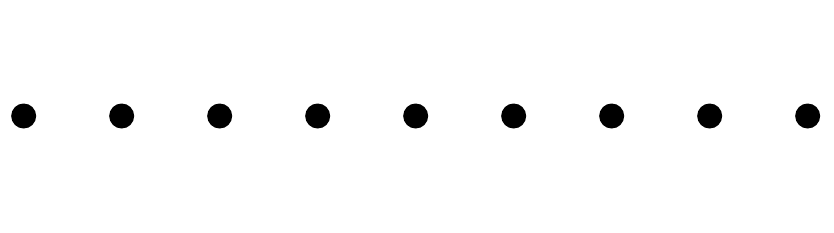} & \includegraphics[scale=.35]{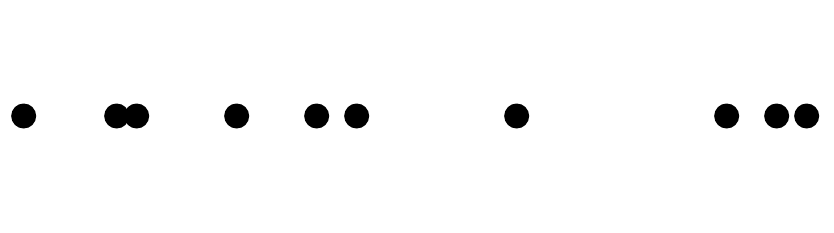}\\ \hline
  \sf{forward jump} & $\sigma(t)\equiv t$ & $\sigma(t)\equiv t+1$ & $\sigma(t)$ varies\\ \hline
  \sf{step size} & $\mu(t)\equiv 0$ & $\mu(t)\equiv 1$ & $\mu(t)$ varies\\ \hline
  \sf{differential operator} & \footnotesize{$\displaystyle\dot{x}(t):=\lim_{h\to 0}{x(t+h)-x(t)\over h}$} & \footnotesize{$\Delta x(t):=x(t+1)-x(t)$} & \footnotesize{$\displaystyle{x^\Delta(t):={x(t+\mu(t))-x(t)\over \mu(t)}}$}\\ \hline
  \sf{canonical equation} & $\dot{x}(t)=Ax(t)$ & $\Delta x(t)=Ax(t)$ & $x^\Delta(t)=Ax(t)$\\ \hline
  \sf{LTI stability region in $\C$} & \rule{0ex}{9em}\includegraphics[scale=.35]{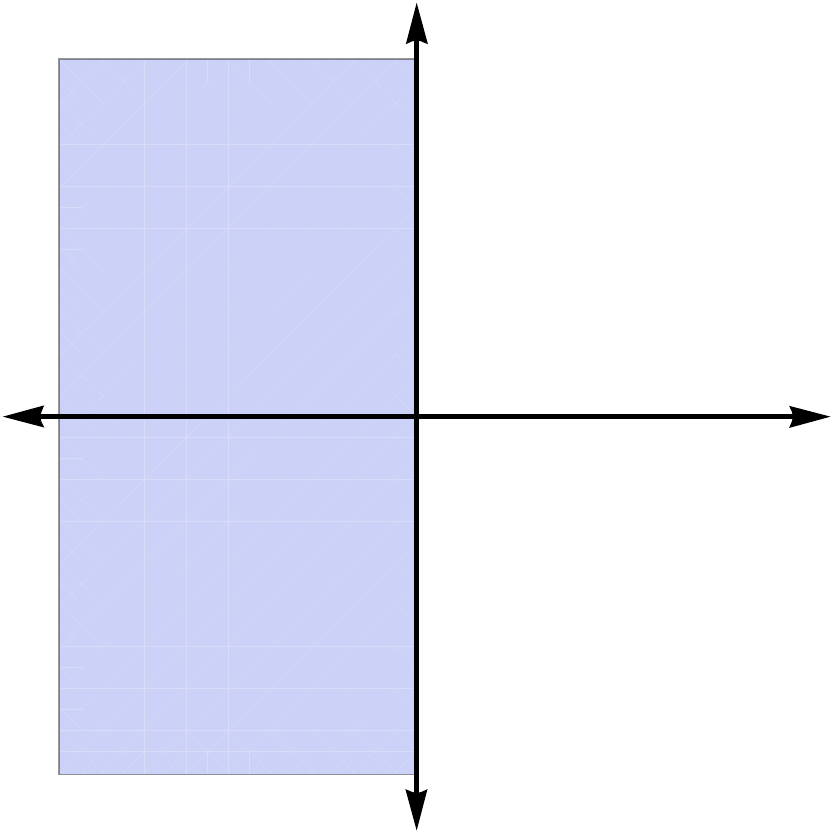}
 & \rule{0ex}{9em}\includegraphics[scale=.35]{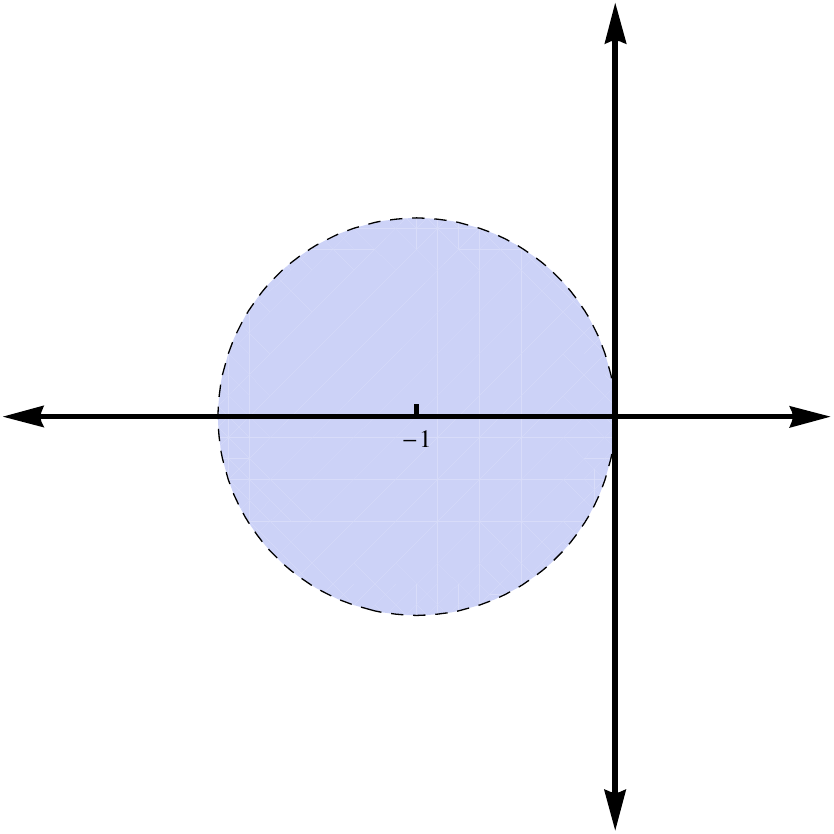} & \rule{0ex}{9em}\includegraphics[scale=.35]{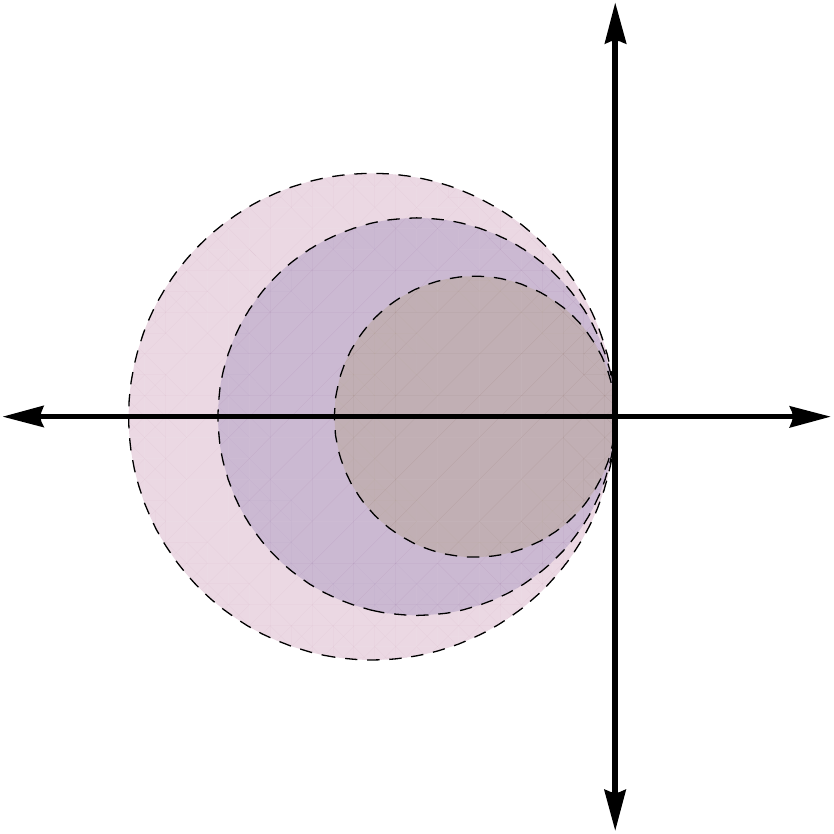}\\ \hline
\end{tabular}
\end{table}

A {\it time scale} $\T$ is any nonempty, (topologically) closed subset of the real numbers $\R$. Thus time scales can be (but are not limited to) any of the usual integer subsets (e.g. $\Z$ or $\N$), the entire real line $\R$, or any combination of discrete points unioned with closed intervals. For example, if $q>1$ is fixed, the {\it quantum time scale} $\overline{q^\Z}$ is defined as
    $$
    \overline{q^\Z}:=\{q^k:k\in\Z\}\cup\{0\}.
    $$
The quantum time scale appears throughout the mathematical physics literature, where the dynamical systems of interest are the $q$-difference equations \cite{Bo,Ca,ChKa}. Another interesting example is the {\it pulse time scale} $\mathbb{P}_{a,b}$ formed by a union of closed intervals each of length $a$ and gap $b$:
    $$
    \mathbb{P}_{a,b}:=\bigcup_k \left[k(a+b),k(a+b)+a\right].
    $$
This time scale is used to study duty cycles of various waveforms. Other examples of interesting time scales include any collection of discrete points sampled from a probability distribution, any sequence of partial sums from a series with positive terms, or even the infamous Cantor set.

The bulk of engineering systems theory to date rests on two time scales, $\R$ and $\Z$ (or more generally $h\Z$, meaning discrete points separated by distance $h$). However, there are occasions when necessity or convenience dictates the use of an alternate time scale. The question of how to approach the study of dynamical systems on time scales then becomes relevant, and in fact the majority of research on time scales so far has focused on expanding and generalizing the vast suite of tools available to the differential and difference equation theorist. We now briefly outline the portions of the time scales theory that are needed for this paper to be as self-contained as is practically possible.

\subsection{The Time Scales Calculus} We now review the time scales calculus needed for the remainder of the paper.

The {\it forward jump operator} is given by $\s(t):=\inf_{s\in\T}\{s>t\}$, while the {\it backward jump operator} is $\rho(t):=\sup_{s\in\T}\{s<t\}$. The {\it graininess function} $\m(t)$ is given by $\m(t):=\s(t)-t$.


A point $t\in\T$ is {\it right-scattered\/} if $\s(t)>t$ and {\it right dense\/} if $\s(t)=t$. A point $t\in\T$ is {\it left-scattered\/} if $\rho(t)<t$ and {\it left dense\/} if $\rho(t)=t$. If $t$ is both left-scattered and right-scattered, we say $t$ is {\it isolated} or {\it discrete}. If $t$ is both left-dense and right-dense, we say $t$ is {\it dense}. The set $\T^\kappa$ is defined as follows: if $\T$ has a left-scattered maximum $m$, then $\T^\kappa=\T-\{m\}$; otherwise, $\T^\kappa=\T$.

For $f:\T\to\R$ and $t\in\T^\kappa$, define $f^\D(t)$ as the number (when it exists), with the property that, for any $\varepsilon > 0$, there exists a neighborhood $U$ of $t$ such that
	\begin{equation}\label{epsdef}
	\left|[f(\sigma(t))-f(s)]-f^\D(t)[\sigma(t)-s]\right|
	\leq\epsilon|\sigma(t)-s|, \quad \forall s\in U.
	\end{equation}
The function $f^\D:\T^\kappa\to\R$ is called the \textit{delta derivative} or the {\it Hilger derivative} of $f$ on $\T^\kappa$. Equivalently, \eqref{epsdef} can be restated to define the $\Delta$-differential operator as
    $$
    x^\Delta(t):={x(\s(t))-x(t)\over \mu(t)},
    $$
where the quotient is taken in the sense that $\m(t)\to 0^+$ when $\m(t)=0$.

\begin{table}
\caption{Differential operators on time scales.}
\label{derivatives}
\renewcommand{\arraystretch}{2}
\begin{tabular}{|c|c|c|}
  \hline
  \sf{time scale} & \sf{differential operator} & \sf{notes}\\ \hline
  $\T$ & $x^\Delta (t)={x(\s(t))-x(t)\over \m(t)}$ & generalized derivative\\ \hline
  $\R$ & $x^\Delta (t)=\lim_{h\to 0}{x(t+h)-x(t)\over h}$ & standard derivative\\ \hline
  $\Z$ & $x^\Delta(t)=\Delta x(t):=x(t+1)-x(t)$ & forward difference\\ \hline
  $h\Z$ & $x^\Delta(t)=\Delta_h x(t):={x(t+h)-x(t)\over h}$ & $h$-forward difference\\ \hline
  $\overline{q^\Z}$ & $x^\Delta(t)=\Delta_q x(t):={x(qt)-x(t)\over (q-1)t}$ & $q$-difference\\ \hline
  ${\mathbb P}_{a,b}$ & $x^\Delta(t)=\begin{cases} {dx\over dt}, & \s(t)=t,\\ {x(t+b)-x(t)\over b}, &\s(t)>t\end{cases}$ & pulse derivative\\ \hline
\end{tabular}
\end{table}

A benefit of this general approach is that the realms of differential equations and difference equations can now be viewed as but special, particular cases of more general {\it dynamic equations on time scales}, i.e. equations involving the delta derivative(s) of some unknown function. See Table~\ref{derivatives}.

\begin{table}
\caption{Integral operators on time scales.}
\label{integrals}
\renewcommand{\arraystretch}{2}
\begin{tabular}{|c|c|c|}
  \hline
  \sf{time scale} & \sf{integral operator} & \sf{notes}\\ \hline
  $\T$ & $\int_\T f(t)\Delta t$ & generalized integral\\ \hline
  $\R$ & $\int_a^b f(t)\Delta t =\int_a^b f(t)\,dt$ & standard Lebesgue integral\\ \hline
  $\Z$ & $\int_a^b f(t)\Delta t =\sum_{t=a}^{b-1} f(t)$ & summation operator\\ \hline
  $h\Z$ & $\int_a^b f(t)\Delta t =\sum_{t=a}^{b-h} f(t)h$ & $h$-summation \\ \hline
  $\overline{q^\Z}$ & $\int_a^b f(t)\Delta t =\sum_{t=a}^{b/q} {f(t)\over (q-1)t}$ & $q$-summation\\ \hline
\end{tabular}
\end{table}

Since the graininess function induces a measure on $\T$, if we consider the Lebesgue integral over $\T$ with respect to the $\mu$-induced measure,
    $$
    \int_{\T}f(t)\,d\mu(t),
    $$
then all of the standard results from measure theory are available \cite{Gu}. In particular, under mild technical assumptions on the integrand, we obtain the set of integral operators in Table~\ref{integrals}.

The upshot here is that the derivative and integral concepts (and all of the concepts in Table~\ref{comparisons}) apply just as readily to {\it any} closed subset of the real line as they do on $\R$ or $\Z$. Our goal to leverage this general framework against wide classes of dynamical and control systems. Progress in this direction has been made in transforms theory \cite{DaGrJaMaRa, MaGrDa}, control \cite{DaGrJaMa, GrDaDC, GrDaDaMa}, dynamic programming \cite{SeSaWu}, and biological models \cite{HoJa1, HoJa2}.

The function $p:\T\to\R$ is \textit{regressive} if $1+\mu(t)p(t)\neq
0$ for all $t\in\T^\kappa$.  We define the related sets
\begin{gather*}
    \mathcal{R}:=\{p:\T\to\R: p\in \textup{C}_\textup{rd}(\T)
    \text{ and }1+\mu(t)p(t)\not=0\  \text{ for all } t\in\T^\kappa\},\\
    \mathcal{R}^+:=\{p\in\mathcal{R}:1+\mu(t)p(t)>0\text{ for all } t\in\T^\kappa\}.
    \end{gather*}

For $p(t)\in\mathcal{R}$, we define the \textit{generalized time
scale exponential function} $e_p(t,t_0)$ as the unique solution to the initial value problem $x^\Delta(t)=p(t)x(t)$, $x(t_0)=1$, which exists when $p\in\mathcal{R}$. See \cite{BoPe1}.

Similarly, the unique solution to the matrix initial value problem $X^\Delta(t)=A(t)X(t)$, $X(t_0)=I$ is called the {\it transition matrix} associated with this system. This solution is denoted by $\Phi_A(t,t_0)$ and exists when $A\in\mathcal{R}$. A matrix is regressive if and only if all of its eigenvalues are in
$\mathcal{R}$.  Equivalently, the matrix $A(t)$ is regressive if and
only if $I+\mu(t)A(t)$ is invertible for all $t\in\T^\kappa$.

\subsection{Stability of Dynamic Systems on Time Scales}

Let $\T$ be a time scale,
unbounded above with bounded graininess. We consider
the dynamic linear system
    \begin{equation}\label{LS}
    x^{\Delta}(t)= A(t) x(t),
    \end{equation}
for $A(t)\in\mathcal{R}(\R^{n\times n})$, and $t\in\T$, where $x^\Delta$ is the
generalized $\Delta$-derivative. Notice that \eqref{LS} reduces to
the familiar systems in \eqref{CLS} and \eqref{DLS} when $\T=\R$ and
$\T=\Z$, respectively. Having seen how Lyapunov's Second Method
allowed us to analyze stability of these systems on the familiar
continuous and discrete domains, we would now like to apply this
method to the analysis of \eqref{LS} defined on an arbitrary time
scale. 

In 2003, P\"{o}tzsche, Siegmund, and
Wirth \cite{PoSiWi} developed spectral criteria for the exponential stability
of \eqref{LS} in the scalar case and in the case that $A(t)\equiv
A$. Then DaCunha \cite{Da2} extended these results by adapting the Second Method of Lyapunov to
the analysis of a certain class of nonautonomous linear systems
(slowly time varying systems) defined on time scales. Here we will
explore and further extend those results, ultimately developing and
solving a \textit{time scale dynamic Lyapunov equation} which
unifies the familiar Lyapunov equations on $\R$ and $\Z$ and is
applicable to a much broader class of systems than those DaCunha
studied.

We begin by giving generalized characterizations of stability for
dynamic linear systems on time scales and then review a few
necessary results from existing theory.

\begin{definition}
    For $t\in\T$, an equilibrium $x=0$ of \eqref{LS} is \textit{Lyapunov stable} or
    \textit{stable in the sense of Lyapunov} if, for every
    $\varepsilon>0$, there exists a $\delta=\delta(\varepsilon)>0$ such that
    if $||x(t_0)||<\delta$, then $||x(t)||<\varepsilon$, for every $t \geq
    t_0$.  An equilibrium $x=0$ of \eqref{LS} is called \textit{uniformly
    stable} if there exists a finite constant $\gamma>0$ such
    that for any $t_0$ and $x(t_0)$, the corresponding solution satisfies
        $$
        ||x(t)||\leq\gamma ||x(t_0)||,\qquad t\geq t_0.
        $$
\end{definition}

\begin{definition}\label{as_T}
    For $t\in\T$, an equilibrium $x=0$ of \eqref{DLS} is \textit{asymptotically
    stable} if it is Lyapunov stable and there exists a $\delta>0$ such
    that if $||x(t_0)||<\delta$, then $\lim_{t\rightarrow\infty} ||x(t)|| =
    0$. Furthermore, an equilibrium $x=0$ of \eqref{DLS} is \textit{uniformly
    asymptotically stable} if it is uniformly stable and if given any
    $\delta>0$ there exists a $T>0$ such that for any
    $t_0$ and $x(t_0)$ the corresponding solution satisfies
        $$
        ||x(t)||\leq \delta ||x(t_0)||,\qquad t\geq t_0+T.
        $$
\end{definition}

\begin{definition}
    For $t\in\T$, an equilibrium $x=0$ of \eqref{DLS} is \textit{exponentially
    stable} if it is asymptotically stable and there exist constants
    $\gamma,\,\lambda,\,\delta>0$ with $-\lambda\in\mathcal{R}^+ $ such that
    if $||x(t_0)||<\delta$, then
    $$
    ||x(t)||\leq\gamma e_{-\lambda}(t, t_0)||x(t_0)||,\qquad t \geq t_0.
    $$
    Furthermore, $x=0$ of \eqref{DLS} is \textit{uniformly
    exponentially stable} if there exist $\gamma,\,\lambda>0$
    with $-\lambda\in\mathcal{R}^+$ such that for any $t_0\text{ and
    }x(t_0)$, the corresponding solution satisfies
    $$
    ||x(t)||\leq||x(t_0)||\gamma
    e_{-\lambda}(t,t_0),\qquad t\geq t_0.
    $$
\end{definition}

These characterizations of stability for system \eqref{LS} are
generalizations of the corresponding characterizations of stability
for systems defined on $\R$ and $\Z$. Specifically, the condition
that $-\lambda\in\mathcal{R}^+$ in the characterization of uniform
exponential stability reduces to $\lambda>0$ and $0 < \lambda < 1$
for $\T=\R$ and $\T=\Z$, respectively.

A necessary and sufficient condition for the
stability of \eqref{LS} (in the scalar case) is given via the
following theorem.
\begin{theorem}\textup{\cite{PoSiWi}}\label{PSW1}
Let $\T$ be a time scale which is unbounded above and let $\lambda
\in \C$.  Then the scalar equation
    $$
    x^{\Delta}(t)=\lambda x(t), \quad x(t_0)=x_0,
    $$
is exponentially stable if and
only if one of the following conditions is satisfied for arbitrary
$t_0 \in \T$:
\begin{itemize}
\item[\textup{(i)}] $\gamma(\lambda):= \displaystyle \limsup_{T \to
\infty}\frac{1}{T-t_0}\int_{t_0}^{T}\displaystyle\lim_{s \searrow
\mu(t)} \frac{\log|1+s\lambda|}{s}\Delta t <0$,
\item[\textup{(ii)}] For every
$T \in \T$, there exists a $t \in \T$ with $t>T$ such that
$1+\mu(t)\lambda=0$,
\end{itemize}
where we use the convention $\log 0=-\infty$ in \textup{(i)}.
\end{theorem}

They then define the {\em set of exponential stability} accordingly, as the collection of
$\lambda\in\C$ satisfying condition (i) or (ii) above.

\begin{definition}\textup{\cite{PoSiWi}}
Given a time scale $\T$ which is unbounded above, define for
arbitrary $t_0 \in \T$,
    $$
    \mathcal{S}_{\C}(\T):=\left\{\lambda \in \C:
    \limsup_{T \to \infty}\frac{1}{T-t_0}\int_{t_0}^{T}\lim_{s \searrow
    \mu(t)} \frac{\log|1+s\lambda|}{s}\Delta t <0\right\}.
    $$
and
    $$
    \mathcal{S}_{\R}(\T):=\{\lambda \in \R: \forall\, T \in \T\ \exists\,
    t \in \T \, \, {\rm with}\, \,  t>T\, \,  {\rm such\, \, that} \, \,
    1+\mu(t)\lambda=0\}.
    $$
Then the {\em set of exponential stability} for the time scale $\T$ is defined by
$$
\mathcal{S}(\T):=\mathcal{S}_{\C}(\T) \cup \mathcal{S}_{\R}(\T).
$$
\end{definition}

Theorem~\ref{PSW1} extends to the time invariant matrix case, $A(t)\equiv A$ as follows.

\begin{theorem}\textup{\cite{PoSiWi}}
Let $\T$ be a time scale that is unbounded above and let $A \in
\R^{n \times n}$ be regressive.  Then the following hold:
\begin{itemize}
\item[\textup{(i)}] If the system \eqref{LS} is exponentially stable,
then $\spec(A)\subset \mathcal{S}_{\C}(\T)$.
\item[\textup{(ii)}] If all
eigenvalues $\lambda$ of $A$ are uniformly regressive, $($i.e.,
$\exists \gamma>0$ such that $\gamma^{-1}\geq |1+\mu(t)\lambda(t)|$,
$t\in\T$$)$ and if $\spec(A) \subset \mathcal{S}_{\C}(\T)$, then
\eqref{LS} is exponentially stable.
\end{itemize}
\end{theorem}

However, this theorem has limitations in practice as the set
$\mathcal{S}$ can be difficult to compute for an arbitrary time
scale. To overcome this, for each fixed $t\in\T$, define the (open) Hilger circle\footnote{More appropriately, the {\it Hilger disk}, but this abuse of language is established in the literature now.} via
    $$
    \H_{\mu(t)}:=\left\{z\in\C_\mu:\left|z+\frac{1}{\mu(t)}\right|<\frac{1}{\mu(t)}\right\}.
    $$
Hoffacker and Gard \cite{GaHo} showed that, if $0\le\mu(t)\le \mu_{\max}$ for all $t\in\T$, then there is a region $\H_{\min}\subset\mathcal{S}_{\C}$, corresponding to $\mu_{\max}$ and given by
    $$
    \H_{\min}:=\left\{z\in\C_{\mu_{\max}}:\left|z+\frac{1}{\mu_{\max}}\right|<\frac{1}{\mu_{\max}}\right\}.
    $$
This yields a static stability region that is more easily
calculable than $\mathcal{S}_{\C}$, albeit more conservative. We conclude that
$\spec{A}\subset\H_{\min}$ is a sufficient (but not necessary) condition for the stability of \eqref{LS} when $A(t)\equiv A$.


Having established the above results for the scalar and autonomous
cases, we turn our attention back to the general case. In order to
extend Lyapunov's Second Method to dynamic equations on time scales,
we define a time scale Lyapunov function and give a unifying
generalization of Theorems \ref{LyapR} and \ref{LyapZ}.
\begin{definition}\label{LyapfnT} A function
$V:\R^n\rightarrow\R$ is called a \textit{generalized} or
\textit{time scale Lyapunov function} for system \eqref{LS} if
    \begin{itemize}
    \item [(i)]$V(x)\geq 0$ with equality if and only if $x=0$, and
    \item [(ii)]$V^\Delta(x(t))\leq 0.$
    \end{itemize}
\end{definition}
\begin{theorem}[Lyapunov's Second Theorem on $\T$, \cite{HoTi,LaSiKa}]\label{LyapT}
    Given system \eqref{DLS} with equilibrium $x=0$, if there exists
    an associated Lyapunov function $V(x)$, then $x=0$ is Lyapunov stable.
    Furthermore, if $V^{\Delta}(x(t))< 0$, then $x=0$ is asymptotically stable.
\end{theorem}

\subsection{The Time Scale Algebraic Lyapunov Equation}
We begin with the quadratic Lyapunov function
candidate $V(x(t))= x^T(t) Px(t)$.  Differentiating with respect
to $t\in\T$ yields
    \begin{align*}
    V^{\Delta}(x(t))&= [x^TPx]^{\Delta}\\
        &= (x^TP)x^{\Delta}+(x^TP)^\Delta x^{\sigma}\\
        &= x^TPA(t)x + x^TA^T(t)P(I+\mu(t)A(t))x\\
        &= x^T[A^T(t)P+(I+\mu(t)A^T(t))PA(t)]x\\
        &= x^T[A^T(t)P+PA(t)+\mu(t)A^T(t)PA(t)]x.
    \end{align*}
If the central quantity satisfies
    $$
    A^T(t)P+PA(t)+\mu(t) A^T(t)PA(t)<0,
    $$
then $V^\Delta(x(t))<0$. Therefore, we seek a solution $P(t)\in\S_n^+$ to the {\it time scale algebraic Lyapunov equation}
    \begin{equation}\label{TSALE}
    A^T(t)P+PA(t)+\mu(t) A^T(t)PA(t)=-M(t),\tag{TSALE}
    \end{equation}
for a given $M(t)\in\S_n^+$.

This algebraic equation unifies the matrix algebraic Lyapunov equations on
$\R$ and $\Z$ discussed earlier: \eqref{TSALE} reduces to \eqref{CALE} on $\T=\R$ and \eqref{DALE} on $\T=\Z$. However, the solutions to \eqref{TSALE} on an arbitrary time scale are fundamentally different than solutions to \eqref{CALE} and \eqref{DALE}---they are generally time varying.

\begin{theorem} [Closed Form Solution of \eqref{TSALE}, \cite{Da2}]\label{TSALEthm}
For each fixed $t\in\T$, define
    \begin{equation*}
    \mathbb{S}_t:=
    \begin{cases}
    \mu(t)\mathbb{N}_0, & \mu(t)\neq 0,\\
    \R_0^+, & \mu(t)=0.
    \end{cases}
    \end{equation*}
The unique solution of \eqref{TSALE} is given by
    \begin{equation}\label{TSALEsol}
    P(t)=\int_{\mathbb{S}_t}\Phi_{A}^T(s,0)M(t)\Phi_{A}(s,0)\Delta s,
    \end{equation}
which converges provided $\lambda\in\H_{\min}$ for all $\lambda\in\spec A$ and all $t\ge T$. Moreover, $P(t)\in\S^+_n$ whenever $M(t)\in\S^+_n$.
\end{theorem}

The upshot here is that even though \eqref{TSALEsol} is a {\it bona fide} solution to \eqref{TSALE}, it does not (in general) lead to constant solutions $P$---and $P$ was assumed to be constant in the Lyapunov candidate at the start of this subsection. Thus, \eqref{TSALE} is not a ``legitimate" Lyapunov equation in the
sense that it is not an appropriate equation to use in a search for
Lyapunov function candidates (even when $A(t)\equiv A$). We are forced to seek a Lyapunov function candidate with a time varying $P$, which we do next.


\subsection{The Time Scale Dynamic Lyapunov Equation}
We begin with the same choice of quadratic Lyapunov function
candidate, $V(x(t))= x^T(t) P(t) x(t)$. Differentiating with respect to $t\in\T$ yields
    \begin{align*}
    V^{\Delta}(x(t))&= [x^T(t) P(t) x(t)]^{\Delta}\\
        &= (x^TP(t))x^\Delta + (x^TP(t))^\Delta x^\sigma\\
        &= x^T P(t)A(t)x + [(x^T)^\Delta P(t)+(x^T)^\sigma P^\Delta(t)]x^\sigma\\
        &= x^T P(t) A(t)x + [x^T A^T P(t) + x^T(I+\mu(t)A(t))^T P^\Delta(t)](I+\mu(t)A(t))x\\
        &= x^T[A^T(t)P(t)+(I+\mu(t)A^T(t))(P^\Delta(t)+P(t)A(t)+\mu(t)P^\Delta(t)A(t))]x\\
        &= x^T[A^T(t)P(t)+P(t)A(t)+\mu(t)A^T(t)P(t)A(t)\\
        &\hskip2in +(I+\mu(t) A^T(t))P^{\Delta}(t)(I+\mu(t)
            A(t))]x.
    \end{align*}
If the central quantity satisfies
    $$
    A^T(t)P(t)+P(t)A(t)+\mu(t) A^T(t)P(t)A(t)+(I+\mu(t) A^T(t))P^{\Delta}(t)(I+\mu(t)A(t))<0,
    $$
then $V^\Delta<0$. Therefore, we seek a solution $P(t)\in\S_n^+$ of the {\it time scale dynamic Lyapunov equation}
    \begin{equation}\label{TSDLE}\tag{TSDLE}
    \begin{aligned}
        A^T(t)P(t)+P(t)A(t)&+\mu(t) A^T(t)P(t)A(t)\\
        &+(I+\mu(t) A^T(t))P^{\Delta}(t)(I+\mu(t) A(t))=-M(t),
    \end{aligned}
    \end{equation}
for a given $M(t)\in\S^+_n$. This equation unifies the matrix differential and difference  Lyapunov equations on $\R$ and $\Z$ discussed earlier: \eqref{TSDLE} reduces to \eqref{CDLE} on $\T=\R$ and to \eqref{DDLE} on $\T=\Z$. Just as importantly, \eqref{TSDLE} also generalizes those types of equations to arbitrary time scales.

\begin{theorem}[Closed Form Solution of \eqref{TSDLE}]\label{TSDLEthm}
The unique solution of
        \begin{equation*}
        \begin{aligned}
        A^T(t)P(t)+P(t)A(t)&+\mu(t) A^T(t)P(t)A(t)\\
        &+(I+\mu(t) A^T(t))P^{\Delta}(t)(I+\mu(t) A(t))=-M(t),\quad P(t_0)=P_0,
        \end{aligned}
        \end{equation*}
is given by
        \begin{equation}\label{TSDLEsol}
            \begin{aligned}
            P(t)&=(\Phi_A^T(t,t_0))^{-1}P(t_0)(\Phi_A(t,t_0))^{-1}\\
                &\hskip.75in-(\Phi_A^T(t,t_0))^{-1}\left[\int^t_{t_0}
                \Phi_A^T(s,t_0)M(s)\Phi_A(s,t_0)\,\Delta s\,\right](\Phi_A(t,t_0))^{-1},
            \end{aligned}
        \end{equation}
where $\Phi_A(t,t_0)$ is the transition matrix for \eqref{DLS}.
\end{theorem}

\begin{proof}
For ease of presentation, we will at times suppress the $t$ dependence in our notation.
Consider
    $$
    A^TP+PA+\mu A^TPA+(I+\mu A^T)P^{\Delta}(I+\mu A)=-M(t).
    $$
Multiply on the left and right by $\Phi_A^T(t,t_0)$ and
$\Phi_A(t,t_0)$, respectively, to obtain
    \begin{align*}
        &\Phi_A^T(t,t_0)\left[A^TP+PA+\mu A^TPA+(I+\mu A^T)P^{\Delta}(I+\mu
                A)\right]\Phi_A(t,t_0)\\
        &\hskip3in=-\Phi_A^T(t,t_0)M(t)\Phi_A(t,t_0).
    \end{align*}
Recognizing the left-hand side of the above equation as a
derivative,
\begin{equation*}
        \left[\Phi_A^T(t,t_0)P(t)\Phi_A(t,t_0)\right]^\Delta
        =-\Phi_A^T(t,t_0)M(t)\Phi_A(t,t_0),
    \end{equation*}
and integrating yields
    \begin{equation*}
        \Phi_A^T(t,t_0)P(t)\Phi_A(t,t_0)-P(t_0)=-\int_{t_0}^t\Phi_A^T(t,t_0)M(t)\Phi_A(t,t_0)\,\Delta
        t.
    \end{equation*}
Rearranging, we obtain
    \begin{align*}
        P(t)&=(\Phi_A^T(t,t_0))^{-1}P(t_0)(\Phi_A(t,t_0))^{-1}\\
            &\hskip.25in-(\Phi_A^T(t,t_0))^{-1}\left[\int^t_{t_0}
                \Phi_A^T(s,t_0)M(s)\Phi_A(s,t_0)\,\Delta
                s\,\right](\Phi_A(t,t_0))^{-1}.
    \end{align*}
\end{proof}

\begin{remark}
We have already seen that \eqref{TSDLE} is a generalized form
unifying the Lyapunov {\it equations} \eqref{CDLE} and \eqref{DDLE} for
systems on $\R$ and $\Z$, respectively, and extending them to arbitrary time domains. Equation \eqref{TSDLEsol} is
also a generalized form unifying the {\it solutions} of \eqref{CDLE} and
\eqref{DDLE} since \eqref{TSDLEsol} becomes \eqref{CDLEsol} on $\R$ and \eqref{DDLEsol} on $\Z$.
\end{remark}

\begin{remark}
At this point, we see how the analysis diverges from that of $\R$
and $\Z$: the solution of \eqref{TSALE} is time varying even when $A(t)\equiv A$ and $M(t)\equiv M$ are constant, since the domain of integration in the solution depends on $\mu(t)$.  Only when operating on time
scales of constant graininess, such as $\R$, $\Z$, and $\T=h\Z$, is
the solution of \eqref{TSALE} constant. On $\R$ and $\Z$,
\eqref{TSALEsol} agrees with the solutions of \eqref{CALE} and
\eqref{DALE} and gives a steady state solution of \eqref{CDLE} and
\eqref{DDLE} as desired. However, on an arbitrary $\T$,
\eqref{TSALEsol} is not a stationary solution of
\eqref{TSDLE} because $P(t)$ is not constant.
\end{remark}

This underscores a crucial difference between algebraic Lyapunov equations on general time scales versus their $\R$ and $\Z$ counterparts: only when the time scale has constant graininess is a solution to an algebraic Lyapunov equation also a (stationary) solution to the dynamic Lyapunov equation.

\subsection{Further Notes on \eqref{TSDLE}}

For a given dynamic linear system \eqref{LS} and choice of initial
condition $P_0$, the closed form solution of \eqref{TSDLE} is known. However, in order for $V(x)=x^TP(t)x$ to be a Lyapunov function, we must know the existence of a solution $P(t)$ in the appropriate space of functions, namely $\S_n^+$. It is not clear that $P(t)$ is in fact positive definite from its form in \eqref{TSDLEsol}. This obstacle is overcome by making a special choice of initial condition.

\begin{theorem}\label{TSDLEICthm}
In Theorem~\ref{TSDLEthm}, if the initial condition is
        \begin{equation}\label{TSDLEIC}
        P(t_0)=P_0:=\int^\infty_{t_0}
            \Phi_A^T(s,t_0)M(s)\Phi_A(s,t_0)\,\Delta s,
        \end{equation}
then \eqref{TSDLEsol} becomes
    \begin{equation}\label{TSDLEICsol}
        P(t)=\int^\infty_{t}\Phi_A^T(s,t)M(s)\Phi_A(s,t)\,\Delta s,
    \end{equation}
and $M(t)\in\S^+_n$ implies $P(t)\in\S^+_n$.
\end{theorem}

\begin{proof}
The solution of \eqref{TSDLE} given in \eqref{TSDLEsol} holds for an
arbitrary choice of $P(t_0)$. Making the specific choice of $P_0$
given in \eqref{TSDLEIC} allows the following simplification:
\begin{align*}
    P(t)&=(\Phi_A^T(t,t_0))^{-1}P(t_0)(\Phi_A(t,t_0))^{-1}\\
         &\hskip.25in-(\Phi_A^T(t,t_0))^{-1}\left[\int^t_{t_0}
            \Phi_A^T(s,t_0)M(s)\Phi_A(s,t_0)\,\Delta
            s\,\right](\Phi_A(t,t_0))^{-1}
\end{align*}
may be written as
\begin{align*}
    P(t)&=(\Phi_A^T(t,t_0))^{-1}\left[\int^\infty_{t_0} \Phi_A^T(s,t_0)M(s)\Phi_A(s,t_0)\,\Delta s\right](\Phi_A(t,t_0))^{-1}\\
         &\hskip.25in-(\Phi_A^T(t,t_0))^{-1}\left[\int^t_{t_0}
            \Phi_A^T(s,t_0)M(s)\Phi_A(s,t_0)\,\Delta
            s\,\right](\Phi_A(t,t_0))^{-1}\\
        &=(\Phi_A^T(t,t_0))^{-1}\left[\int^\infty_{t_0} \Phi_A^T(s,t_0)M(s)\Phi_A(s,t_0)\,\Delta
        s\right.\\
         &\hskip.25in-\left.\int^t_{t_0} \Phi_A^T(s,t_0)M(s)\Phi_A(s,t_0)\,\Delta
            s\,\right] (\Phi_A(t,t_0))^{-1}\\
        &=(\Phi_A^T(t,t_0))^{-1}\left[\int^\infty_{t}\Phi_A^T(s,t_0)M(s)\Phi_A(s,t_0)\,\Delta
            s\,\right](\Phi_A(t,t_0))^{-1}\\
        &=(\Phi_A^T(t,t_0))^{-1}\left[\int^\infty_{t}\Phi_A^T(t,t_0)\Phi_A^T(s,t)M(s)\Phi_A(s,t)\Phi_A(t,t_0)\,\Delta
            s\,\right](\Phi_A(t,t_0))^{-1}\\
        &=\int^\infty_{t}\Phi_A^T(s,t)M(s)\Phi_A(s,t)\,\Delta
            s.
\end{align*}
From the resulting quadratic form of $P(t)$, we see
$M(t)\in\S^+_n$ implies $P(t)\in\S^+_n$.
\end{proof}

\begin{remark}
The choice of initial condition given in Theorem~\ref{TSDLEICthm} is
necessary.  Choosing any other initial condition results in
\eqref{TSDLEsol} being unbounded. For $P(t_0)=P_0+\varepsilon$,
\eqref{TSDLEsol} becomes
\begin{align*}
    P(t)&=(\Phi_A^T(t,t_0))^{-1}(P_0+\varepsilon)(\Phi_A(t,t_0))^{-1}\\
         &\hskip.25in-(\Phi_A^T(t,t_0))^{-1}\left[\int^t_{t_0}
             \Phi_A^T(s,t_0)M(s)\Phi_A(s,t_0)\,\Delta
             s\,\right](\Phi_A(t,t_0))^{-1}\\
        &=(\Phi_A^T(t,t_0))^{-1}\varepsilon(\Phi_A(t,t_0))^{-1}+(\Phi_A^T(t,t_0))^{-1}(P_0)(\Phi_A(t,t_0))^{-1}\\
         &\hskip.25in-(\Phi_A^T(t,t_0))^{-1}\left[\int^t_{t_0}
             \Phi_A^T(s,t_0)M(s)\Phi_A(s,t_0)\,\Delta
             s\,\right](\Phi_A(t,t_0))^{-1}\\
        &=(\Phi_A^T(t,t_0))^{-1}\varepsilon(\Phi_A(t,t_0))^{-1}
         +\int^\infty_{t}\Phi_A^T(s,t)M(s)\Phi_A(s,t)\,\Delta s.
\end{align*}
For stable, constant $A$, $\Phi_A(s,t)$ is a decaying matrix exponential. Thus
while the integral term may converge, the quantity
$\|(\Phi_A^T(t,t_0))^{-1}\varepsilon(\Phi_A(t,t_0))^{-1}\|\rightarrow
\infty$ as $t\rightarrow\infty$.
\end{remark}

\begin{remark}
With this choice of initial condition, $P(t)$ reduces to a useful
form, especially for aligning our results with the literature, at least in the following sense.  If $\T$ has constant graininess, $t_0=0$, and $M$ is constant, then the initial matrix $P_0$ in \eqref{TSDLEIC} is in fact a (constant) solution of \eqref{TSALE} which is in turn a stationary solution of \eqref{TSDLE}. Therefore, precisely this choice of initial matrix (under the assumptions above) produces steady-state solutions of \eqref{TSDLE} from its algebraic counterpart \eqref{TSALE}. This is what happens on $\R$ and $\Z$ but fails on general time scales.
\end{remark}

\begin{remark}
On the other hand, the form of \eqref{TSDLEICsol} allows us to deduce $P(t)\in\S^+_n$ whenever $M(t)\in\S^+_n$, which is essential if we want to apply Theorem~\ref{LyapT}. For such a solution $P(t)$, we know $V(x):=x^T P(t) x>0$ and $V^\Delta(x)<0$. Therefore, the existence of \eqref{TSDLEICsol} is sufficient to establish asymptotic stability of \eqref{LS}.
\end{remark}


\end{document}